 \documentclass[tablecaption=bottom,wcp]{jmlr} 

\usepackage{booktabs}
\usepackage[load-configurations=version-1]{siunitx} 

\theorembodyfont{\upshape}
\theoremheaderfont{\scshape}
\theorempostheader{:}
\theoremsep{\newline}

\DeclareMathOperator*{\argmin}{arg\,min}

\usepackage{amsmath}

\numberwithin{equation}{section}

\jmlrproceedings{AABI 2019}{2nd Symposium on Advances in Approximate Bayesian Inference, 2019}

\title[MMD-Bayes: Robust Bayesian Estimation via Maximum Mean Discrepancy]{MMD-Bayes: Robust Bayesian Estimation via Maximum Mean Discrepancy}

  \author{\Name{Badr-Eddine Ch\'erief-Abdellatif} \Email{badr.eddine.cherief.abdellatif@ensae.fr}\\
  \addr CREST, ENSAE, Institut Polytechnique de Paris
  \AND
  \Name{Pierre Alquier} \Email{pierrealain.alquier@riken.jp}\\
  \addr RIKEN Center for AI Project, Tokyo, Japan
  }

\begin{document}

\maketitle

\begin{abstract}
In some misspecified settings, the posterior distribution in Bayesian statistics may lead to inconsistent estimates. To fix this issue, it has been suggested to replace the likelihood by a pseudo-likelihood, that is the exponential of a loss function enjoying suitable robustness properties. In this paper, we build a pseudo-likelihood based on the Maximum Mean Discrepancy, defined via an embedding of probability distributions into a reproducing kernel Hilbert space. We show that this MMD-Bayes posterior is consistent and robust to model misspecification. As the posterior obtained in this way might be intractable, we also prove that reasonable variational approximations of this posterior enjoy the same properties. We provide details on a stochastic gradient algorithm to compute these variational approximations. Numerical simulations indeed suggest that our estimator is more robust to misspecification than the ones based on the likelihood.
\end{abstract}

\section{Introduction}
\label{sec:intro}

Bayesian methods are very popular in statistics and machine learning as they provide a natural way to model uncertainty. Some subjective prior distribution $\pi$ is updated using the negative log-likelihood $\ell_n$ via Bayes' rule to give the posterior $\pi_n(\theta) \propto \pi(\theta) \exp(-\ell_n(\theta))$. Nevertheless, the classical Bayesian methodology is not robust to model misspecification. There are many cases where the posterior is not consistent \citep{barron1999consistency,grunwaldmisspecification}, and there is a need to develop methodologies yielding robust estimates. \textcolor{black}{A way to fix this problem is to replace the log-likelihood $\ell_n$ by a relevant risk measure. This idea is at the core of the PAC-Bayes theory~\citep{MR2483528} and Gibbs posteriors~\citep{SM2018}; its connection with Bayesian principles are discussed in~\cite{Bissiri2016UpdateBayes}.} \cite{Knoblauch2019} builds a general representation of Bayesian inference in the spirit of \cite{Bissiri2016UpdateBayes} and extends the representation to the approximate inference case. In particular, the use of a robust divergence has been shown to provide an estimator that is robust to misspecification \citep{Knoblauch2019}. For instance, \cite{HookerRobustBayes2014} investigated the case of Hellinger-based divergences, \cite{GhoshBasuRobustBayes2016}, \cite{FutamiRobustBayes2017}, and \cite{NakagawaRobustBayes2019} used robust $\beta$- and $\gamma$-divergences, while \cite{Catoni2012}, \cite{BaraudBirge2017RhoBayes} and \cite{Holland2019} replaced the logarithm of the log-likelihood by wisely chosen bounded functions. Refer to \cite{JewsonRobustBayes2018} for a complete survey on robust divergence-based Bayes inference.

In this paper, we consider the Maximum Mean Discrepancy (MMD) as the alternative loss used in Bayes' formula, leading to a pseudo-posterior that we shall call MMD-Bayes in the following. MMD is built upon an embedding of distributions into a reproducing kernel Hilbert space (RKHS) that generalizes the original feature map to probability measures, and allows to apply tools from kernel methods in parametric estimation. Our MMD-Bayes posterior \textcolor{black}{is related to the kernel-based posteriors in~\cite{KernelBayes2013}, \cite{Park2016} and \cite{RidgwayKABC17}, even though it is different}. More recently, \cite{Briol2019} introduced a frequentist minimum distance estimator based on the MMD distance, that is shown to be consistent and robust to small deviations from the model. We show that our MMD-Bayes retains the same properties, i.e is consistent at the minimax optimal rate of convergence as the minimum MMD estimator, and is also robust to misspecification, including data contamination and outliers. Moreover, we show that these guarantees are still valid when considering a tractable approximation of the MMD-Bayes via variational inference, and we support our theoretical results with experiments showing that our approximation is robust to outliers for various estimation problems. All the proofs are deferred to the appendix.

\section{Framework and definitions}
\label{sec:notations}

Let us introduce the background and theoretical tools required to understand the rest of the paper. We consider in a measurable space $\big( \mathbb{X},\mathcal{X} \big)$ a collection of $n$ independent and identically distributed (i.i.d) random variables $X_1,...,X_n \sim P_0$ where $P_0$ is the generating distribution. We index a statistical model $\{ P_{\theta}/ \theta \in \Theta \}$ by a parameter space $\Theta$, without necessarily assuming that the true distribution $P_0$ belongs to the model.


Let us consider some integrally strictly positive definite kernel $k$ \footnote{ This means that the positive definite kernel satisfies $\mathbb{E}_{X,Y\sim P}[k(X,Y)]\ne0$ for any distribution \textcolor{black}{$P$}. This includes the Gaussian kernel $k(x,y) = \exp(-\|x-y\|^2/\gamma^2)$. \textcolor{black}{For this property, and the properties of MMD discussed in this section, we refer the reader to~\cite{Fuku2017}}.} bounded by a positive constant, \textcolor{black}{say $1$}. We then denote the associated RKHS $({\mathcal{H}_{k}},\langle\cdot,\cdot\rangle_{\mathcal{H}_{k}})$ satisfying the reproducing property $f(x)=\langle f,  k(x,\cdot)\rangle_{\mathcal{H}_{k}}$ for any $f \in {\mathcal{H}_{k}}$ and any $x \in \mathbb{X}$. We define the notion of \textit{kernel mean embedding}, a Hilbert space embedding that maps probability distributions into the RKHS $\mathcal{H}_{k}$. Given a distribution $P$, the kernel mean embedding $\mu_P \in {\mathcal{H}_{k}}$ is
$$
\mu_P(\cdot) := \mathbb{E}_{X\sim P}[k(X,\cdot)] \in {\mathcal{H}_{k}} .
$$
Then we define the MMD between two probability distributions $P$ and $Q$ simply as the distance in $\mathcal{H}_k$ between their kernel mean embeddings:
$$
\mathbb{D}_k(P,Q) = \| \mu_P - \mu_Q \|_{\mathcal{H}_{k}} .
$$
Under the assumptions we made on the kernel, the kernel mean embedding is injective and the maximum mean discrepancy is a metric, see \cite{Briol2019}. We motivate the use of MMD as a robust metric in Appendix \ref{apd:MMD-is-robust}.  

In this paper, we adopt a Bayesian approach. We introduce a prior distribution $\pi$ over the parameter space $\Theta$ equipped with some sigma-algebra. Then we define our pseudo-Bayesian distribution $\pi_n^\beta$ given a prior $\pi$ on $\Theta$:
\[
\pi_n^\beta(d\theta) \propto \exp\bigg({-\beta \cdot \mathbb{D}^{\textcolor{black}{2}}_k(P_\theta,\hat{P}_n)\bigg)}\pi(d\theta),\]
where $\hat{P}_n=(1/n)\sum_{i=1}^n \delta_{X_i}$ is the empirical measure and $\beta>0$ is a temperature parameter. 

\section{Theoretical analysis of MMD-Bayes}
\label{sec:consistency}

In this section, we show that the MMD-Bayes is consistent when the true distribution belongs to the model, and \textcolor{black}{is robust to} misspecification.

To obtain the concentration of posterior distributions in models that contain the generating distribution, \cite{ghosal2000convergence} introduced the so-called \textit{prior mass condition} that requires the prior to put enough mass to some neighborhood (in Kullback-Leibler divergence) of the true distribution. This condition was widely studied since then for more general pseudo-posterior distributions \citep{bhattacharya2016bayesian,Tempered,cherief2018consistency}. Unfortunately, this prior mass condition is (by definition) restricted to cases when the model is well-specified or at least when the true distribution is in a very close neighborhood of the model. We formulate here a robust version of the prior mass condition which is based on a neighborhood of an approximation $\theta^*$ of the true parameter instead of the true parameter itself. The following condition is suited to the MMD metric, recovers the usual prior mass condition when the model is well-specified and still makes sense in misspecified cases with potentially large deviations to the model assumptions:

\vspace{0.2cm}

\noindent
\textbf{\textit{Prior mass condition:}}
\textit{Let us denote $\theta^*=\argmin_{\theta \in \Theta} \mathbb{D}_k\left(P_{\theta},P_0\right)$ and its neighborhood $\mathcal{B}_n = \{\theta \in \Theta / \mathbb{D}_k\left(P_{\theta},P_{\theta^*} \right) \leq \textcolor{black}{n^{-1/2}} \}$. Then $(\pi,\beta)$ is said to satisfy the prior mass condition $C(\pi,\beta)$ when $\pi(\mathcal{B}_n) \geq e^{-\beta/\textcolor{black}{n}}$.}

\vspace{0.2cm}

In the usual Bayesian setting, the computation of the prior mass is a major difficulty \citep{ghosal2000convergence}, and it can be hard to know whether the prior mass condition is satisfied or not. Nevertheless, here the condition does not only hold on the prior distribution $\pi$ but also on the temperature parameter $\beta$. Hence, it is always possible to choose $\beta$ large enough so that the prior mass condition is satisfied. We refer the reader to Appendix \ref{apd:priorcomputation} for an example of computation of such a prior mass and valid values of $\beta$. The following theorem expressed as a generalization bound shows that the MMD-Bayes posterior distribution is robust to misspecification under the robust prior mass condition. Note that the rate $n^{-1/2}$ is exactly the one obtained by the frequentist MMD estimator of \cite{Briol2019} and is minimax optimal \citep{Tol2017}:

\begin{theorem}
\label{theorem:mmd:mis}
Under the prior mass condition $C(\pi,\beta)$:
\begin{equation}
\label{formula-well}
\mathbb{E} \left[ \int \mathbb{D}^{\textcolor{black}{2}}_k\left( P_{\theta} ,P_0 \right) \pi^{\beta}_n({\rm d}\theta) \right]  \leq \textcolor{black}{8} \inf_{\theta\in\Theta}  \mathbb{D}_k^{\textcolor{black}{2}}\left(P_{\theta},P_0 \right) + \textcolor{black}{\frac{16}{n}} .
\end{equation}
\end{theorem}



The second theorem investigates concentration of the MMD-Bayes posterior in the well-specified case. It shows that the prior mass condition $C(\pi,\beta)$ ensures that the MMD-Bayes concentrates to $P_0$ at the minimax rate $n^{-1/2}$:
\begin{theorem}
\label{theorem:mmd:well}
Let us consider a well-specified model. Then under the prior mass condition $C(\pi,\beta)$, we have in probability for any $M_n \rightarrow +\infty$: 
\begin{equation}
\label{concentration-well}
\pi_n^\beta \bigg(\mathbb{D}_k( P_{\theta}, P_0 )>M_n \cdot n^{-1/2} \bigg) \xrightarrow[n\rightarrow +\infty]{} 0 .
\end{equation}
\end{theorem}

Note that we obtain the concentration to the true distribution $P_0=P_{\theta^*}$ at the minimax rate $n^{-1/2}$ for well-specified models.

\section{Variational inference}
\label{sec:VI}

Unfortunately, the MMD-Bayes is not tractable in complex models. In this section, we provide an efficient implementation of the MMD-Bayes based on VI retaining the same theoretical properties. Given a variational set of tractable distributions $\mathcal{F}$, we define the variational approximation of $\pi_n^\beta$ as the closest approximation (in KL divergence) to the target MMD posterior:
$$
\tilde{\pi}_n^\beta = \underset{\rho \in \mathcal{F}}{\arg\min} \, 
\textnormal{KL}(\rho\|\pi_n^\beta) .
$$

Under similar conditions to those in Theorems \ref{theorem:mmd:mis} and \ref{theorem:mmd:well}, $\tilde{\pi}_n^\beta$ is guaranteed to be $n^{-1/2}$-consistent as the MMD-Bayes. Most works ensuring the consistency or the concentration of variational approximations of posterior distributions use the \textit{extended prior mass condition}, an extension of the prior mass condition that applies to variational approximations rather than on the distributions they \textcolor{black}{ approximate  \citep{alquier2016properties,Tempered,Plage,cherief2018consistency,cherief2018consistency2,cherief2019deep}}. Here, we extend our previous prior mass condition to variational approximations but also to misspecification. In addition to the prior mass condition inspired from \cite{ghosal2000convergence}, the variational set $\mathcal{F}$ must contain probability distributions that are concentrated around the best approximation $P_{\theta^*}$. This robust extended prior mass condition can be formulated as follows:

\vspace{0.2cm}

\noindent
\textbf{\textit{Assumption :}}
\textit{We assume that there exists a distribution $\rho_{n} \in \mathcal{F}$ such that}:
\begin{equation}
\label{extendedpriormass}
  \int \mathbb{D}_k^{\textcolor{black}{2}} (P_{\theta},P_{\theta^*}) \rho_{n}(d\theta) \leq \textcolor{black}{\frac{1}{n}} \hspace{0.2cm} \textnormal{and} \hspace{0.2cm} 
  \textrm{KL}(\rho_{n}\|\pi) \leq \textcolor{black}{\frac{\beta}{n} } .
\end{equation}

\begin{remark}
\label{rmk-1} 
When the restriction of $\pi$ to the MMD-ball $\mathcal{B}_n$ centered at $\theta^*$ of radius $n^{-1/2}$ belongs to $\mathcal{F}$, then Assumption \eqref{extendedpriormass} becomes the \textcolor{black}{standard} robust prior mass condition, i.e. $\pi(\mathcal{B}_n) \geq \textcolor{black}{e^{-\beta /n }}$. In particular, when \textcolor{black}{$\mathcal{F} $ is the set of all probability measures -- that is, in the case where there is no variational approximation --} then we recover the \textcolor{black}{standard condition}.
\end{remark}

Now, we can state the following theorem for variational approximations:

\begin{theorem}
\label{theorem:vi:mis}
Under the extended prior mass condition \eqref{extendedpriormass}, 
\begin{equation}
\label{formula-well-vi}
\mathbb{E} \left[ \int \mathbb{D}^{\textcolor{black}{2}}_k\left( P_{\theta} ,P_0 \right) \tilde{\pi}^{\beta}_n({\rm d}\theta) \right]  \leq \textcolor{black}{8} \inf_{\theta\in\Theta}  \mathbb{D}_k^{\textcolor{black}{2}}\left(P_{\theta},P_0 \right) + \textcolor{black}{\frac{16}{n}} .
\end{equation}
Moreover, if the model is well-specified, then under the prior mass condition $C(\pi,\beta)$, we have in probability for any $M_n \rightarrow +\infty$: 
\begin{equation}
\label{concentration-well-vi}
\tilde{\pi}_n^\beta \bigg(\mathbb{D}_k( P_{\theta}, P_0 )>M_n \cdot n^{-1/2} \bigg) \xrightarrow[n\rightarrow +\infty]{} 0 .
\end{equation}
\end{theorem}

\section{Numerical experiments}
\label{sec:example}

In this section, we show that the variational approximation is robust in practice when estimating a Gaussian mean and a uniform distribution in the presence of outliers. We consider here a $d$-dimensional parametric model and a Gaussian mean-field variational set $\mathcal{F}=\{\mathcal{N}(m,\textnormal{diag}(s^2))/m\in\mathcal{M}, s\in\mathcal{S}\}$, $\mathcal{M}\subset\mathbb{R}^d, \mathcal{S}\subset\mathbb{R}^d_{>0}$, using componentwise multiplication. Inspired from the stochastic gradient descent of \cite{Roy2015}, \cite{li15} and \cite{Briol2019} based on a U-statistic approximation of the MMD criterion, we design a stochastic gradient descent that is suited to our variational objective. The algorithm is described in details in Appendix \ref{apd:algorithm}. 

We perform short simulations to provide empirical support to our theoretical results. Indeed, we consider the problem of Gaussian mean estimation in the presence of outliers. The experiment consists in randomly sampling $n=200$ i.i.d observations from a Gaussian distribution $\mathcal{N}(2,1)$ but some corrupted observations are replaced by samples from a standard Cauchy distribution $\mathcal{C}(0,1)$. The fraction of outliers used was ranging from $0$ to $0.20$ with a step-size of $0.025$. We repeated each experiment $100$ times and considered the square root of the mean square error (MSE). The plots we obtained demonstrate that our method performs comparably to the componentwise median (MED) and even better as the number of outliers increases, and clearly outperforms the maximum likelihood estimator (MLE). We also conducted the simulations for multidimensional Gaussians and for the robust estimation of the location parameter of a uniform distribution. \textcolor{black}{We refer the reader to Appendix \ref{apd:simulations} for more details on these simulations.}

\begin{figure}[htbp]
\floatconts
  {fig:nodes}
  {}
  {\includegraphics[width=1.0\linewidth]{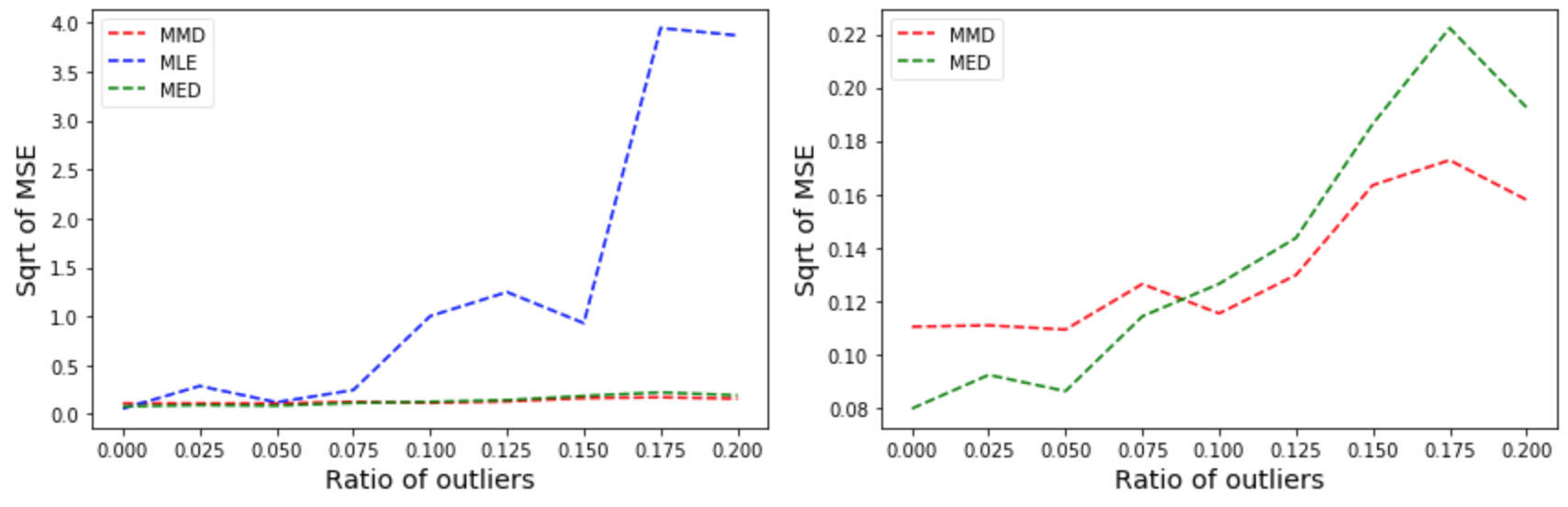}}
  \vspace{-0.5cm}
  {Figure 1 - Comparison of the square root of the MSE for the MMD estimator, the MLE and the median in the robust Gaussian mean estimation problem for various values of the proportion of outliers. The MMD estimator is the mean of the variational approximation.}
\end{figure}


\vspace{-0.5cm}
\section{Conclusion}
\label{sec:conclusion}

In this paper, we showed that the MMD-Bayes posterior concentrates at the minimax convergence rate and is robust to model misspecification. We also proved that reasonable variational approximations of this posterior retain the same properties, and we proposed a stochastic gradient algorithm to compute such approximations that we supported with numerical simulations. An interesting future line of research would be to investigate if the i.i.d assumption can be relaxed and if the MMD-based estimator is also robust to dependency in the data.

\newpage

\bibliographystyle{apalike}

\newpage

\appendix

\section{Proof of Theorem \ref{theorem:mmd:mis}.}\label{apd:proof-thm-mis}

In order to prove Theorem \ref{theorem:mmd:mis}, we first need two preliminary lemmas. The first one ensures the convergence of the empirical measure $\hat{P}_n$ to the true distribution $P_0$ (in MMD distance $\mathbb{D}_k$) at the minimax rate $n^{-1/2}$, and which is an expectation variant of Lemma 1 in \cite{Briol2019} that holds with high probability:

\begin{lemma}
 \label{lemma:mmd:1}
 We have
 $$ \mathbb{E} \left[ \mathbb{D}_k^{\textcolor{black}{2}} \left( \hat{P}_n,P_0 \right) \right] \leq \textcolor{black}{\frac{1}{n}}. $$
\end{lemma}

\begin{proof}
\begin{align*}
\mathbb{E} \left[ \mathbb{D}_k^2\left( \hat{P}_n,P_0 \right) \right] 
& = \mathbb{E} \left[  \left\| \frac{1}{n}\sum_{i=1}^{n}\left[ k(X_i,\cdot) - \mu_{P_0} \right] \right\|_{{\mathcal{H}_{k}}}^2 \right]
\\
& = \frac{1}{n^2}  \mathbb{E} \left[ \sum_{i=1}^n \left\| k(X_i,\cdot) - \mu_{P_0} \right\|_{{\mathcal{H}_{k}}}^2  + 2 \sum_{1\leq i<j \leq n} \left<k(X_i,\cdot) - \mu_{P_0},k(X_j,\cdot) -\mu_{P_0} \right>_{\mathcal{H}_k} \right]
\\
& \leq \frac{1}{n^2} \left( n  + 2 \sum_{1\leq i<j \leq n} 0 \right)
 = \frac{1}{n} \textcolor{black}{.}
\end{align*}
\end{proof}
\textcolor{black}{ The rate $n^{-1/2}$ is known to be minimax in this case, see Theorem 1 in~\cite{Tol2017}}.

The second lemma is a simple triangle-like inequality that will be widely used throughout the proofs of the paper:
\begin{lemma}
 \label{lemma:mmd:2}
 We have for any distributions $P$, $P'$ and $Q$:
 $$  \mathbb{D}_k^{\textcolor{black}{2}} \left( P,P' \right) \leq 2  \mathbb{D}_k^{\textcolor{black}{2}} \left( P,Q \right) + 2  \mathbb{D}_k^{\textcolor{black}{2}} \left( Q,P' \right)  . $$
\end{lemma}

\begin{proof}
The chain of inequalities follow directly from the triangle inequality and inequality $2ab\leq a^2+b^2$.
\begin{align*}
\mathbb{D}_k^2\left( P,P' \right) 
& \leq \bigg( \mathbb{D}_k \left( P,Q \right) + \mathbb{D}_k \left( Q,P' \right) \bigg)^2
\\
& =  \mathbb{D}_k^{\textcolor{black}{2}} \left( P,Q \right) + \mathbb{D}_k^{\textcolor{black}{2}} \left( Q,P' \right) + 2 \mathbb{D}_k \left( P,Q \right) \mathbb{D}_k \left( Q,P' \right)
\\
& \leq  \mathbb{D}_k^{\textcolor{black}{2}} \left( P,Q \right) + \mathbb{D}_k^{\textcolor{black}{2}} \left( Q,P' \right) + \mathbb{D}_k^{\textcolor{black}{2}} \left( P,Q \right) + \mathbb{D}_k^{\textcolor{black}{2}} \left( Q,P' \right)
\\
& = 2  \mathbb{D}_k^{\textcolor{black}{2}} \left( P,Q \right) + 2  \mathbb{D}_k^{\textcolor{black}{2}} \left( Q,P' \right) .
\end{align*}
\end{proof}

Let us come back to the proof of Theorem \ref{theorem:mmd:mis}. An important point is that the MMD-Bayes can also be defined using an argmin over the set $\mathcal{M}^1_+(\Theta)$ of all probability distributions absolutely continuous with respect to $\pi$ and the Kullback-Leibler divergence KL$(\cdot\|\cdot)$:
\[
\pi_n^\beta = \underset{\rho \in \mathcal{M}^1_+(\Theta)}{\arg\min} \, \bigg\{
 \int \mathbb{D}_k^{\textcolor{black}{2}}(P_{\theta},\hat{P}_n) \rho({\rm d}\theta) + \frac{ \textnormal{KL}(\rho\|\pi) }{ \beta }\bigg\}.
\]
\textcolor{black}{This is an immediate consequence of Donsker and Varadhan's variational inequality, see e.g ~\cite{MR2483528}}. Using the triangle inequality, Lemma~\ref{lemma:mmd:1}, Lemma~\ref{lemma:mmd:2} for different settings of $P$, $P'$ and $Q$, and Jensen's inequality:
\begin{align*}
\mathbb{E}\left[ \int \mathbb{D}_k^{\textcolor{black}{2}}\left( P_{\theta}  ,P_0 \right) \pi^{\beta}_n({\rm d}\theta) \right] & \leq \textcolor{black}{2} \mathbb{E}\left[ \int \mathbb{D}_k^{\textcolor{black}{2}}\left( P_{\theta}  ,\hat{P}_n \right) \pi^{\beta}_n({\rm d}\theta) \right] + 2 \mathbb{E} \left[ \mathbb{D}_k^{\textcolor{black}{2}} \left( \hat{P}_n,P_0 \right) \right]
\\
& \leq \textcolor{black}{2} \mathbb{E}\left[ \int \mathbb{D}_k^{\textcolor{black}{2}}\left( P_{\theta}  ,\hat{P}_n \right) \pi^{\beta}_n({\rm d}\theta) \right] + \textcolor{black}{\frac{2}{n}}
\\
& \leq \textcolor{black}{2} \mathbb{E}\left[  \int \mathbb{D}_k^{\textcolor{black}{2}}\left( P_{\theta}  ,\hat{P}_n \right) \pi^{\beta}_n({\rm d}\theta) + \frac{ \textnormal{KL}(\pi^\beta_n\|\pi) }{ \beta } \right] + \textcolor{black}{ \frac{2}{n} }
\\
& = \textcolor{black}{2} \mathbb{E}\left[  \inf_\rho \left\{  \int \mathbb{D}_k^{\textcolor{black}{2}}\left( P_{\theta}  ,\hat{P}_n \right) \rho({\rm d}\theta) + \frac{ \textnormal{KL}(\rho\|\pi) }{ \beta }  \right\} \right]+ \textcolor{black}{\frac{2}{n} }
\\
& \leq  \textcolor{black}{2} \inf_\rho  \mathbb{E}\left[  \int \mathbb{D}_k^{\textcolor{black}{2}}\left( P_{\theta}  ,\hat{P}_n \right) \rho({\rm d}\theta) + \frac{ \textnormal{KL}(\rho\|\pi) }{ \beta } \right] + \textcolor{black}{ \frac{2}{n}} ,
\end{align*}
which gives, using Lemma~\ref{lemma:mmd:1} and the triangle inequality again:
\begin{align*}
\mathbb{E}\left[ \int \mathbb{D}^{\textcolor{black}{2}}_k\left( P_{\theta}  ,P_0 \right) \pi^{\beta}_n({\rm d}\theta) \right] & \leq \textcolor{black}{2} \inf_\rho  \mathbb{E}\left[  \int \mathbb{D}^{\textcolor{black}{2}}_k\left( P_{\theta}  ,\hat{P}_n \right) \rho({\rm d}\theta) + \frac{ \textnormal{KL}(\rho\|\pi) }{ \beta } \right] + \textcolor{black}{ \frac{2}{n} }
\\
 & \leq \textcolor{black}{2} \inf_\rho  \mathbb{E}\left[  \int \mathbb{D}^{\textcolor{black}{2}}_k\left( P_{\theta}  ,P_0 \right) \rho({\rm d}\theta) + \frac{ \textnormal{KL}(\rho\|\pi) }{ \beta } \right] + 4 \mathbb{E} \left[ \mathbb{D}_k^{\textcolor{black}{2}} \left( \hat{P}_n,P_0 \right) \right] + \textcolor{black}{ \frac{2}{n} }
\\
& = \textcolor{black}{2} \inf_\rho \mathbb{E} \left[ \textcolor{black}{2} \int \mathbb{D}^{\textcolor{black}{2}}_k\left( P_{\theta}  ,P_0 \right) \rho({\rm d}\theta) + \frac{ \textnormal{KL}(\rho\|\pi) }{ \beta } \right]+ \textcolor{black}{\frac{6}{n}}
\\
& \leq \textcolor{black}{8} \mathbb{D}_k^{\textcolor{black}{2}}\left( P_{\theta^*}, P_0 \right) + \textcolor{black}{2} \inf_\rho \mathbb{E} \left[ \textcolor{black}{4} \int \mathbb{D}_k^{\textcolor{black}{2}}\left( P_{\theta}  ,P_{\theta^*} \right) \rho({\rm d}\theta) + \frac{ \textnormal{KL}(\rho\|\pi) }{ \beta } \right]+ \textcolor{black}{\frac{6}{n}} 
\end{align*}
We remind that $\theta^*=\argmin_{\theta \in \Theta} \mathbb{D}_k\left(P_{\theta},P_0\right)$.

This bound can be formulated in the following way when $\rho$ is chosen to be equal to $\pi$ restricted to $\mathcal{B}_n$ :
$$ \mathbb{E}\left[ \int \mathbb{D}_k^{\textcolor{black}{2}}\left( P_{\theta}  ,P_0 \right) \pi^{\beta}_n({\rm d}\theta) \right]  \leq \textcolor{black}{8} \inf_{\theta\in\Theta}  \mathbb{D}_k^{\textcolor{black}{2}}\left(P_{\theta},P_0 \right)  + \textcolor{black}{\frac{8}{n}} + \textcolor{black}{2}\frac{-\log \pi (B)}{\beta} + \textcolor{black}{\frac{6}{n}} . $$
Finally, as soon as the prior mass condition $C(\pi,\beta)$ is satisfied, we get:
$$ \mathbb{E}\left[ \int \mathbb{D}_k^{\textcolor{black}{2}}\left( P_{\theta}  ,P_0 \right) \pi^{\beta}_n({\rm d}\theta) \right]   \leq \textcolor{black}{8} \inf_{\theta\in\Theta}  \mathbb{D}_k^{\textcolor{black}{2}} \left(P_{\theta},P_0 \right)  +  \textcolor{black}{\frac{16}{n}} . $$

\section{Proof of Theorem \ref{theorem:mmd:well}.}\label{apd:proof-thm-well}

In case of well-specification, Formula \eqref{formula-well} simply becomes according to Jensen's inequality:
$$ \mathbb{E}\left[ \int \mathbb{D}_k\left( P_{\theta}  ,P_0 \right) \pi^{\beta}_n({\rm d}\theta) \right] \leq \sqrt{\mathbb{E}\left[ \int \mathbb{D}_k^2\left( P_{\theta}  ,P_0 \right) \pi^{\beta}_n({\rm d}\theta) \right]}  \leq \sqrt{\frac{16}{n}} = \frac{4}{\sqrt{n}} . $$
Hence, it is sufficient to show that the inequality above implies the concentration of the MMD-Bayes to the true distribution. This is a simple consequence of Markov's inequality. Indeed, for any $M_n \rightarrow +\infty$:
$$
\mathbb{E} \bigg[ \pi_n^\beta \bigg(\mathbb{D}_k( P_{\theta}, P_0 )>M_n \cdot n^{-1/2} \bigg) \bigg] \leq \frac{\mathbb{E} \bigg[ \int \mathbb{D}_k( P_{\theta}, P_0 ) \pi_n^\beta(d\theta) \bigg]}{M_n \cdot n^{-1/2} } \leq \frac{4 n^{-1/2} }{M_n \cdot n^{-1/2} } \xrightarrow[n\rightarrow +\infty]{} 0,
$$
which guarantees the convergence in mean of $\pi_n^\beta \big(\mathbb{D}_k( P_{\theta}, P_0 )>M_n \cdot n^{-1/2} \big)$ to $0$, which leads to the convergence in probability of $\pi_n^\beta \big(\mathbb{D}_k( P_{\theta}, P_0 )>M_n \cdot n^{-1/2} \big) $ to $0$, i.e. the concentration of MMD-Bayes to $P_0$ at rate $n^{-1/2}$.

\section{Proof of theorem \ref{theorem:vi:mis}.}\label{apd:proof-thm-vi}

Formula \eqref{formula-well-vi} can be proven easily as for the proof of Theorem \ref{theorem:mmd:mis}. Indeed, we use the expression of the variational approximation of the MMD-Bayes using an argmin over the set $\mathcal{F}$:
\[
\tilde{\pi}_n^\beta = \underset{\rho \in \mathcal{F}}{\arg\min} \, \bigg\{
 \int \mathbb{D}_k^{\textcolor{black}{2}}(P_{\theta},\hat{P}_n) \rho({\rm d}\theta) + \frac{ \textnormal{KL}(\rho\|\pi) }{ \beta }\bigg\}.
\]
\textcolor{black}{
This is yet an application of Donsker and Varadhan's lemma.} Then, as previously:
\begin{align*}
\mathbb{E}\left[ \int \mathbb{D}_k^{\textcolor{black}{2}}\left( P_{\theta}  ,P_0 \right) \tilde{\pi}^{\beta}_n({\rm d}\theta) \right] & \leq \mathbb{E}\left[ \int \mathbb{D}_k^{\textcolor{black}{2}}\left( P_{\theta}  ,\hat{P}_n \right) \tilde{\pi}^{\beta}_n({\rm d}\theta) \right] +\textcolor{black}{\frac{2}{n}}\text{ by Lemma~\ref{lemma:mmd:1}} 
\\
& \leq \textcolor{black}{2} \mathbb{E}\left[  \int \mathbb{D}^{\textcolor{black}{2}}_k\left( P_{\theta}  ,\hat{P}_n \right) \tilde{\pi}^{\beta}_n({\rm d}\theta) + \frac{ \textnormal{KL}(\pi^\beta_n\|\pi) }{ \beta } \right] + \textcolor{black}{\frac{2}{n}}
\\
& = \textcolor{black}{2} \mathbb{E}\left[  \inf_\rho \left\{  \int \mathbb{D}^{\textcolor{black}{2}}_k\left( P_{\theta}  ,\hat{P}_n \right) \rho({\rm d}\theta) + \frac{ \textnormal{KL}(\rho\|\pi) }{ \beta }  \right\} \right]+\textcolor{black}{\frac{2}{n}} 
\\
& \leq \textcolor{black}{2} \inf_\rho  \mathbb{E}\left[  \int \mathbb{D}^{\textcolor{black}{2}}_k\left( P_{\theta}  ,\hat{P}_n \right) \rho({\rm d}\theta) + \frac{ \textnormal{KL}(\rho\|\pi) }{ \beta } \right] + \textcolor{black}{\frac{2}{n}}
\\
& \leq \textcolor{black}{2} \inf_\rho \mathbb{E} \left[ \textcolor{black}{2} \int \mathbb{D}^{\textcolor{black}{2}}_k\left( P_{\theta}  ,P_0 \right) \rho({\rm d}\theta) + \frac{ \textnormal{KL}(\rho\|\pi) }{ \beta } \right]+ \textcolor{black}{\frac{6}{n}}
\\
& \leq \textcolor{black}{8} \mathbb{D}_k^{\textcolor{black}{2}} \left( P_{\theta^*}, P_0 \right) + \textcolor{black}{2} \inf_\rho \mathbb{E} \left[ \textcolor{black}{4} \int \mathbb{D}_k^{\textcolor{black}{2}} \left( P_{\theta}  ,P_{\theta^*} \right) \rho({\rm d}\theta) + \frac{ \textnormal{KL}(\rho\|\pi) }{ \beta } \right]+ \textcolor{black}{\frac{6}{n}} .
\end{align*}

Hence, under the extended prior mass condition \eqref{extendedpriormass}, we have directly:
$$ \mathbb{E}\left[ \int \mathbb{D}_k^{\textcolor{black}{2}}\left( P_{\theta}  ,P_0 \right) \tilde{\pi}^{\beta}_n({\rm d}\theta) \right]   \leq \textcolor{black}{8}\inf_{\theta\in\Theta}  \mathbb{D}_k^{\textcolor{black}{2}}\left(P_{\theta},P_0 \right)  + \textcolor{black}{\frac{16}{n}} . $$

The proof of Formula \eqref{concentration-well-vi} follows the lines of the proof of Theorem \ref{theorem:mmd:well}.

\section{An example of robustness of the MMD distance.}\label{apd:MMD-is-robust}

In this appendix, we try to give some intuition on the choice of MMD-Bayes rather than the classical regular Bayesian distribution. To do so, we show a simple misspecified example for which the MMD distance is more suited than the classical Kullback-Leibler (KL) divergence used in the Bayes rule in the definition of the classical Bayesian posterior.

We consider the Huber's contamination model described as follows. We observe a collection of random variables $X_1,...,X_n$. There are unobserved i.i.d random variables $Z_1,...,Z_n \sim \text{Ber}(\epsilon)$ and a distribution $Q$, such that the distribution of $X_i$ given $Z_i=0$ is a Gaussian $\mathcal{N}(\theta^0,\sigma^2)$ where the distribution of $X_i$ given $Z_i=0$ is $Q$. The observations $X_i$'s are independent. This is equivalent to considering a true distribution $P_0=(1-\epsilon)\mathcal{N}(\theta^0,\sigma^2)+\epsilon Q$. Here, $\epsilon\in(0,1/2)$ is the contamination rate, $\sigma^2$ is a known variance and $Q$ is the contamination distribution that is taken here as $\mathcal{N}(\theta_c,\sigma^2)$, where $\theta_c$ is the mean of the corrupted observations. The true parameter of interest is $\theta^0$ and the model is composed Gaussian distributions $\{P_\theta=\mathcal{N}(\theta,\sigma^2)/\theta \in \mathbb{R}^d\}$. The goal in this appendix is to show that we exactly recover the true parameter $\theta^0$ with the minimizer of the MMD distance to the true distribution $P_0$, whereas it is not the case with the KL divergence. We use a Gaussian kernel $k(x,y) = \exp(-\|x-y\|^2/\gamma^2)$.

\subsection*{Computation of the MMD distance to the true distribution:}

We have remind that $P_\theta = \mathcal{N}(\theta,\sigma^2 I_d)$ where $\theta\in\Theta=\mathbb{R}^d$. For independent $X$ and $Y$ following respectively $P_\theta$ and $P_{\theta'}$, we get $(X-Y)\sim \mathcal{N}(\theta-\theta',\sigma^2 I_d)$. Hence,
$$ \frac{X-Y}{\sqrt{2\sigma^2}} \sim \mathcal{N}\left( \frac{\theta-\theta'}{\sqrt{2\sigma^2 }},I_d \right) $$
and the square of this random variable is a noncentral chi-square random variable:
$$ \frac{\| X-Y \|^2}{2\sigma^2} \sim \chi^2\left( d,  \frac{\| \theta-\theta' \|^2}{2\sigma^2} \right). $$
It is known that for $U\sim \chi^2(d,m)$, we have $ \mathbb{E}[\exp(tU)] = \exp(mt/(1-2t)) /(1-2t)^{d/2} $, and then $t=-(2\sigma^2)/\gamma^2$ gives:
\begin{align*}
\langle \mu_{P_\theta},\mu_{P_{\theta'}}\rangle_{\mathcal{H}_k} & = \mathbb{E}_{X\sim P_\theta,Y\sim P_{\theta'}}\left[\exp\left(-\frac{\|X-Y\|^2}{\gamma^2}\right)\right] \\
& = \left(\frac{\gamma^2}{4\sigma^2 + \gamma^2}\right)^{\frac{d}{2}} \exp\left( -\frac{\|\theta-\theta'\|^2}{4\sigma^2 + \gamma^2} \right) .
\end{align*}
Thus,
$$
\langle \mu_{P_\theta},\mu_{P_\theta}\rangle_{\mathcal{H}_k} =  \left(\frac{\gamma^2}{4\sigma^2+\gamma^2}\right)^{\frac{d}{2}} ,
$$
\begin{align*}
\langle\mu_{P_\theta},\mu_{P_0}\rangle_{\mathcal{H}_k} & = (1-\epsilon) \langle\mu_{P_\theta},\mu_{P_{\theta^0}}\rangle_{\mathcal{H}_k} + \epsilon \langle\mu_{P_\theta},\mu_{P_{\theta_c}}\rangle_{\mathcal{H}_k} \\
& = (1-\epsilon) \left(\frac{\gamma^2}{4\sigma^2 + \gamma^2}\right)^{\frac{d}{2}} \exp\left( -\frac{\|\theta-\theta^0\|^2}{4\sigma^2 + \gamma^2} \right) \\
& \quad \quad \quad \quad \quad \quad \quad \quad \quad \quad \quad \quad + \epsilon \left(\frac{\gamma^2}{4\sigma^2 + \gamma^2}\right)^{\frac{d}{2}} \exp\left( -\frac{\|\theta-\theta_c\|^2}{4\sigma^2 + \gamma^2} \right) ,
\end{align*}
and
\begin{align*}
\langle\mu_{P_0},\mu_{P_0}\rangle_{\mathcal{H}_k} & = (1-\epsilon)^2 \langle\mu_{P_{\theta^0}},\mu_{P_{\theta^0}}\rangle_{\mathcal{H}_k} + 2 \epsilon(1-\epsilon) \langle\mu_{P_{\theta^0}},\mu_{P_{\theta_c}}\rangle_{\mathcal{H}_k} + \epsilon^2 \langle\mu_{P_{\theta_c}},\mu_{P_{\theta_c}}\rangle_{\mathcal{H}_k} \\
& = (1-\epsilon)^2  \left(\frac{\gamma^2}{4\sigma^2 + \gamma^2}\right)^{\frac{d}{2}} + \epsilon^2  \left(\frac{\gamma^2}{4\sigma^2 + \gamma^2}\right)^{\frac{d}{2}} \\
& \quad \quad \quad \quad \quad \quad \quad \quad \quad \quad \quad +  2 \epsilon(1-\epsilon)  \left(\frac{\gamma^2}{4\sigma^2 + \gamma^2}\right)^{\frac{d}{2}} \exp\left( -\frac{\|\theta^0-\theta_c\|^2}{4\sigma^2 + \gamma^2} \right) \\
& = \bigg(1 -2 \epsilon(1-\epsilon)\bigg)  \left(\frac{\gamma^2}{4\sigma^2 + \gamma^2}\right)^{\frac{d}{2}} \\
& \quad \quad \quad \quad \quad \quad \quad \quad \quad \quad \quad  + 2 \epsilon(1-\epsilon) \left(\frac{\gamma^2}{4\sigma^2 + \gamma^2}\right)^{\frac{d}{2}} \exp\left( -\frac{\|\theta^0-\theta_c\|^2}{4\sigma^2 + \gamma^2} \right).
\end{align*}

Hence
\begin{align*}
\mathbb{D}_k^2\left(P_0,P_{\theta} \right) & = \|\mu_{P_\theta}-\mu_{P_0}\|_{\mathcal{H}_k}^2 = \langle\mu_{P_\theta},\mu_{P_\theta}\rangle_{\mathcal{H}_k} -2  \langle\mu_{P_\theta},\mu_{P_0}\rangle_{\mathcal{H}_k} +  \langle \mu_{P_0},\mu_{P_0}\rangle_{\mathcal{H}_k} \\
& = 2\bigg(1 - \epsilon(1-\epsilon)\bigg)  \left(\frac{\gamma^2}{4\sigma^2 + \gamma^2}\right)^{\frac{d}{2}} - 2 \epsilon \left(\frac{\gamma^2}{4\sigma^2 + \gamma^2}\right)^{\frac{d}{2}} \exp\left( -\frac{\|\theta-\theta_c\|^2}{4\sigma^2 + \gamma^2} \right) \\
& \quad \quad \quad \quad  \quad \quad \quad \quad  + 2 \epsilon(1-\epsilon) \left(\frac{\gamma^2}{4\sigma^2 + \gamma^2}\right)^{\frac{d}{2}} \exp\left( -\frac{\|\theta^0-\theta_c\|^2}{4\sigma^2 + \gamma^2} \right) \\
& \quad \quad \quad \quad \quad \quad \quad \quad \quad \quad \quad - 2 (1-\epsilon) \left(\frac{\gamma^2}{4\sigma^2 + \gamma^2}\right)^{\frac{d}{2}} \exp\left( -\frac{\|\theta-\theta^0\|^2}{4\sigma^2 + \gamma^2} \right) \\
& = 2(1-\epsilon) \left(\frac{\gamma^2}{4\sigma^2 + \gamma^2}\right)^{\frac{d}{2}} \bigg[ 1 - \exp\left( -\frac{\|\theta-\theta^0\|^2}{4\sigma^2 + \gamma^2} \right) \bigg] \\
& \quad \quad \quad \quad \quad \quad \quad \quad + 2 \epsilon \left(\frac{\gamma^2}{4\sigma^2 + \gamma^2}\right)^{\frac{d}{2}} \bigg[ 1 - \exp\left( -\frac{\|\theta-\theta_c\|^2}{4\sigma^2 + \gamma^2} \right) \bigg] \\
& \quad \quad \quad \quad \quad \quad \quad \quad \quad \quad \quad - 2 \epsilon(1-\epsilon) \left(\frac{\gamma^2}{4\sigma^2 + \gamma^2}\right)^{\frac{d}{2}} \bigg[ 1 - \exp\left( -\frac{\|\theta^0-\theta_c\|^2}{4\sigma^2 + \gamma^2} \right) \bigg]  .
\end{align*}

Hence, the minimizer of $\mathbb{D}_k\left(P_0,P_{\theta} \right)$ w.r.t $\theta$, i.e the maximizer of:
$$
(1-\epsilon) \exp\left(-\frac{\|\theta -\theta^0\|^2}{4\sigma^2+\gamma^2}\right) + \epsilon \exp\left(-\frac{\|\theta-\theta_c\|^2}{4\sigma^2+\gamma^2}\right) .
$$
is $\theta^0$ itself as $\epsilon\leq1/2$.

\subsection*{Computation of the KL divergence to the true distribution:}

In this case, easy computations lead for any $\theta$ to:
\begin{align*}
\textnormal{KL}(P_0\|P_\theta) & = \textnormal{KL}\bigg( (1-\epsilon)\mathcal{N}(\theta^0,\sigma^2)+\epsilon \mathcal{N}(\theta_c,\sigma^2) \| \mathcal{N}(\theta,\sigma^2) \bigg) \\
& = C + (1-\epsilon) H(\theta^0\|\theta) + \epsilon H(\theta_c\|\theta) \\
& = C + \frac{d\log(2\pi\sigma^2)}{2} + \frac{d\sigma^2}{2} + (1-\epsilon) \frac{\|\theta-\theta^0\|^2}{2\sigma^2} + \epsilon \frac{\|\theta-\theta_c\|^2}{2\sigma^2} ,
\end{align*}
where
\begin{align*}
H(\theta'\|\theta) & = - \int\log\big(\mathcal{N}(x|\theta,\sigma^2)\big) \mathcal{N}(x|\theta',\sigma^2) dx \\
& = \frac{d\log(2\pi\sigma^2)}{2} + \frac{d\sigma^2}{2} + \frac{\|\theta-\theta'\|^2}{2\sigma^2}
\end{align*}
is the cross-entropy of $P_\theta$ and $P_{\theta'}$, and 
\begin{align*}
C & = (1-\epsilon)\int\log\bigg((1-\epsilon)\mathcal{N}(x|\theta^0,\sigma^2)+\epsilon \mathcal{N}(x|\theta_c,\sigma^2)\bigg) \mathcal{N}(x|\theta^0,\sigma^2) dx \\
& \quad \quad \quad + \epsilon \int\log\bigg((1-\epsilon)\mathcal{N}(x|\theta^0,\sigma^2)+\epsilon \mathcal{N}(x|\theta_c,\sigma^2)\bigg) \mathcal{N}(x|\theta_c,\sigma^2) dx,
\end{align*}
where $\mathcal{N}(x|m,\sigma^2)$ is the probability density function of $\mathcal{N}(m,\sigma^2)$ evaluated at $x$.

Hence, the minimizer of $\textnormal{KL}\left(P_0\|P_{\theta} \right)$ w.r.t $\theta$, i.e the minimizer of:
$$
(1-\epsilon) \|\theta-\theta^0\|^2 + \epsilon \|\theta-\theta_c\|^2 .
$$
is $(1-\epsilon)\theta^0+\epsilon\theta_c$, which can be far away from $\theta^0$ in situations when the corrupted mean $\theta_c$ is very far from the true parameter $\theta^0$.

\section{An example of computation of a robust prior mass.}\label{apd:priorcomputation}

In this appendix, we tackle the computation of a prior mass in the Gaussian mean estimation problem, and we show that it leads to a wide range of values of $\beta$ satisfying the prior mass condition $C(\pi,\beta)$ for a standard normal prior $\pi$. 

We recall that the prior mass condition $C(\pi,\beta)$ is satisfied as soon as there exists a function $f$ such that:
$$
\beta \geq - \log \pi(\mathcal{B}_n) \textcolor{black}{n} .
$$
In practice, lower bounds of the form $\pi(\mathcal{B}_n)\geq \mathcal{L}e^{-f(\theta^*)}$ naturally appear when computing the prior mass $\pi(\mathcal{B}_n)$. Only $f(\theta^*)$ depends on the parameter $\theta^*$ corresponding to the best approximation in the model of the true distribution in the MMD sense, that is the true parameter itself when the model is well-specified. Hence, it is sufficient to choose a value of the temperature parameter $\beta \geq \big( f(\theta^*) - \log \mathcal{L}\big)\textcolor{black}{n} $ in order to obtain the prior mass condition.

We conduct the computation in a misspecified case, where we assume that a proportion $1-\epsilon$ of the observations are sampled i.i.d from a $\sigma^2$-variate Gaussian distribution of interest $P_{\theta^0}$, but that the remaining observations are corrupted and can take any arbitrary value. We consider the model of Gaussian distributions $\{P_\theta=\mathcal{N}(\theta,\sigma^2)/\theta\in\mathbb{R}^d\}$. This adversarial contamination model is more general than Huber's contamination model presented in Appendix \ref{apd:MMD-is-robust}. Note that when $\epsilon=0$, then the model is well-specified and the distribution of interest $P_{\theta^0}$ is also the true distribution $P_0$. We use the Gaussian kernel $k(x,y) = \exp(-\|x-y\|^2/\gamma^2)$ and the standard normal prior $\pi = \mathcal{N}(0,I_d)$.

We write the inequality defining parameters $\theta$ belonging to $\mathcal{B}_n$:
\begin{equation}
\label{inequalityB}
\mathbb{D}^2_k\left(P_{\theta^*},P_{\theta} \right) \leq n^{-1} .
\end{equation}
Note that when the model is well-specified, the we get $\theta^*=\theta^0$.

According to derivations performed in Appendix \ref{apd:MMD-is-robust}, we have for any $\theta$:
\begin{align*}
\mathbb{D}_k^2\left(P_{\theta},P_{\theta^*} \right) & = \langle\mu_{P_\theta},\mu_{P_\theta}\rangle_{\mathcal{H}_k} -2  \langle\mu_{P_\theta},\mu_{P_{\theta^*}}\rangle_{\mathcal{H}_k} +  \langle \mu_{P_{\theta^*}},\mu_{P_{\theta^*}}\rangle_{\mathcal{H}_k} \\
& = 2 \left(\frac{\gamma^2}{4\sigma^2 + \gamma^2}\right)^{\frac{d}{2}} \left[ 1 - \exp\left( -\frac{\|\theta-\theta^*\|^2}{4\sigma^2 + \gamma^2} \right) \right] .
\end{align*}
Hence, Inequality \eqref{inequalityB} is equivalent to:
$$
2 \left(\frac{\gamma^2}{4\sigma^2 + \gamma^2}\right)^{\frac{d}{2}} \left[ 1 - \exp\left( -\frac{\|\theta-\theta^*\|^2}{4\sigma^2 + \gamma^2} \right) \right] \leq \frac{1}{n}
$$
i.e to
$$
1-\frac{1}{2n}\left(1+\frac{4\sigma^2}{\gamma^2}\right)^{\frac{d}{2}} \leq \exp\left(-\frac{\|\theta-\theta^*\|^2}{4\sigma^2 + \gamma^2}\right)
$$
We denote $s_n=\sqrt{\frac{4\sigma^2 + \gamma^2}{2n}} \left(1+\frac{4\sigma^2}{\gamma^2}\right)^{\frac{d}{4}}$ and $\mathbb{B}(\theta,s_n)$ the ball of radius $s_n$ and centered at $\theta$. Let us compute the prior mass of $\mathcal{B}_n$:
\begin{align*}
\pi(\mathcal{B}_n) & = \pi\bigg(1-\frac{1}{2n}\left(1+\frac{4\sigma^2}{\gamma^2}\right)^{\frac{d}{2}} \leq \exp\left(-\frac{\|\theta-\theta^*\|^2}{4\sigma^2 + \gamma^2}\right)\bigg) \\
& \geq \pi\bigg(1-\frac{1}{2n}\left(1+\frac{4\sigma^2}{\gamma^2}\right)^{\frac{d}{2}} \leq 1 -\frac{\|\theta-\theta^*\|^2}{4\sigma^2 + \gamma^2}\bigg)  \quad \quad \textnormal{using inequality $e^{-x}\geq1-x$} \\
& = \pi\bigg(\|\theta-\theta^*\|^2 \leq (4\sigma^2 + \gamma^2) \frac{1}{2n} \left(1+\frac{4\sigma^2}{\gamma^2}\right)^{\frac{d}{2}} \bigg) \\
& = \pi\big(\theta\in\mathbb{B}(\theta^*,s_n)\big) \\
& = \int_{\mathbb{B}(\theta^*,s_n)} (2\pi)^{-d/2} e^{{-\|\theta\|^2}/{2}} d\theta .
\end{align*}
Actually, the point that minimizes $\theta \rightarrow e^{{-\|\theta\|^2}/{2}}$ on $\mathbb{B}(\theta^*,s_n)$ is $\theta^* (1+s_n/\|\theta^*\|)$. Thus:
\begin{align*}
\pi(\mathcal{B}_n) & \geq \int_{\mathbb{B}(\theta^*,s_n)} (2\pi)^{-d/2} \exp\bigg(\frac{-\|\theta\|^2}{2}\bigg) d\theta \\
& \geq (2\pi)^{-d/2} \exp\bigg(\frac{-(\|\theta^*\|+s_n)^2}{2}\bigg) {vol}\big(\mathbb{B}(\theta^*,s_n)\big).
\end{align*}
We recall the formula of the volume of the d-dimensional ball:
$$
{vol}\big(\mathbb{B}(\theta^*,s_n)\big)=\frac{\pi^{d/2}}{\Gamma(d/2+1)} s_n^d .
$$ 
Hence:
$$
\pi(\mathcal{B}_n) \geq \frac{\big(4\sigma^2 + \gamma^2\big)^{\frac{d}{2}} \left(1+\frac{4\sigma^2}{\gamma^2}\right)^{\frac{d^2}{4}}}{\Gamma(d/2+1)} \exp\bigg(-\frac{1}{2}\bigg\{\|\theta^*\|+\sqrt{\frac{4\sigma^2 + \gamma^2}{2n}} \left(1+\frac{4\sigma^2}{\gamma^2}\right)^{\frac{d}{4}}\bigg\}^2\bigg) \frac{1}{n^{d/2}} .
$$
As could be expected for a standard normal prior, the larger the value of $\|\theta^*\|$, the smaller can be the prior mass.

We denote 
$$
\mathcal{L} = \frac{\big(4\sigma^2 + \gamma^2\big)^{\frac{d}{2}} \left(1+\frac{4\sigma^2}{\gamma^2}\right)^{\frac{d^2}{4}}}{\Gamma(d/2+1)} \cdot \frac{1}{n^{d/2}} 
$$
and
$$
f(x)=\frac{1}{2}\bigg\{\|x\|+\sqrt{\frac{4\sigma^2 + \gamma^2}{2n}} \left(1+\frac{4\sigma^2}{\gamma^2}\right)^{\frac{d}{4}}\bigg\}^2 
$$
so that $\pi(\mathcal{B}_n) \geq \mathcal{L}e^{-f(\theta^*)}$.

Hence, for the standard normal prior $\pi$, values of $\beta$ leading to consistency of the MMD-Bayes are:
\begin{align*}
\beta & \geq  ( f(\theta^*) - \log \mathcal{L} ) \textcolor{black}{n}  \\
& = \frac{\textcolor{black}{n}}{2}\bigg\{\|\theta^*\|+\sqrt{\frac{4\sigma^2 + \gamma^2}{2n}} \left(1+\frac{4\sigma^2}{\gamma^2}\right)^{\frac{d}{4}}\bigg\}^2 + \frac{d \textcolor{black}{n}\log n}{2}  \\ 
& \quad \quad \quad \quad \quad \quad \quad - \frac{d \textcolor{black}{n}}{2} \log(4\sigma^2+\gamma^2) - \frac{d^2\textcolor{black}{n}}{4} \log\bigg(1+\frac{4\sigma^2}{\gamma^2}\bigg) +\textcolor{black}{n} \log \Gamma(d/2+1) .
\end{align*}

In particular, when $\gamma^2$ is of order $d$, then using Stirling's approximation, we get a lower bound on the valid values of $\beta$ of order (up to a logarithmic factor):
$$
\textcolor{black}{n} \max\big(\|\theta^*\|^2,d\big) \lesssim \beta.
$$


\section{Computation of the extended prior mass.}\label{apd:extendedpriormass}

The computation of Condition \eqref{extendedpriormass} is of major interest. We investigate here the case of a Gaussian model $P_\theta=\mathcal{N}(\theta,\sigma^2)$, a Gaussian mean-field variational approximation $\mathcal{F}=\{\mathcal{N}(m,\textnormal{diag}(s^2))/m\in\mathbb{R}^d, s\in\mathbb{R}^d_{>0}\}$, a standard Gaussian prior $\pi=\mathcal{N}(0,1)$ and a Gaussian kernel $k(x,y) = \exp(-\|x-y\|^2/\gamma^2)$.

Let us define $\rho_n=\mathcal{N}\big(\theta^*,s^2 I_d\big)$ where $s^2=\frac{4\sigma^2 + \gamma^2}{2dn} \left(1+\frac{4\sigma^2}{\gamma^2}\right)^{\frac{d}{2}}$. Then:
\begin{align*}
\textrm{KL}(\rho_{n}\|\pi) & = \frac{1}{2} \sum_{j=1}^d \bigg\{ \theta_j^{*2} + s^2 - \log(s^2) - 1 \bigg\} \\
& = \frac{4\sigma^2 + \gamma^2}{2dn} \left(1+\frac{4\sigma^2}{\gamma^2}\right)^{\frac{d}{2}} + \frac{d\log(2dn) + \|\theta^*\|^2 - d - d\log(4\sigma^2+\gamma^2)}{2} \\
& \quad \quad \quad \quad \quad \quad \quad \quad \quad \quad \quad \quad \quad \quad \quad \quad \quad \quad \quad \quad - \frac{d^2}{4} \log\left(1+\frac{4\sigma^2}{\gamma^2}\right) ,
\end{align*}
and
\begin{align*}
\int \mathbb{D}_k^{\textcolor{black}{2}}(P_{\theta^*},P_{\theta}) \rho_{n}(d\theta) & = 2 \left(\frac{\gamma^2}{4\sigma^2 + \gamma^2}\right)^{\frac{d}{2}} \bigg( 1 - \int \exp\left(-\frac{\|\theta-\theta^*\|^2}{4\sigma^2+\gamma^2}\right) \rho_{n}(d\theta) \bigg) \\
& = 2 \left(\frac{\gamma^2}{4\sigma^2 + \gamma^2}\right)^{\frac{d}{2}} \bigg( 1 - \int e^{-\|\theta\|^2} \mathcal{N}\left(d\theta\bigg|0,\frac{s^2}{{4\sigma^2+\gamma^2}} I_d\right) \bigg) \\
& = 2 \left(\frac{\gamma^2}{4\sigma^2 + \gamma^2}\right)^{\frac{d}{2}} \bigg( 1 - \text{Det}\left(I_d+2\frac{s^2}{{4\sigma^2+\gamma^2}}I_d\right)^{-1/2}  \bigg) \\
& = 2 \left(\frac{\gamma^2}{4\sigma^2 + \gamma^2}\right)^{\frac{d}{2}} \bigg( 1 - \prod_{j=1}^d \left(1+\frac{2s^2}{{4\sigma^2+\gamma^2}}\right)^{-1/2}  \bigg) \\
& = 2 \left(\frac{\gamma^2}{4\sigma^2 + \gamma^2}\right)^{\frac{d}{2}} \bigg( 1 - \left(1+\left(1+\frac{4\sigma^2}{\gamma^2}\right)^{\frac{d}{2}} \frac{1}{{dn}}\right)^{-d/2}  \bigg) \\
& \leq 2 \left(\frac{\gamma^2}{4\sigma^2 + \gamma^2}\right)^{\frac{d}{2}}  \bigg( 1 -\left(1-\frac{d}{2}\left(1+\frac{4\sigma^2}{\gamma^2}\right)^{\frac{d}{2}} \frac{1}{dn}\right) \bigg) = \frac{1}{n} .
\end{align*}

Hence, the robust extended prior mass condition is satisfied as soon as 
\begin{align*}
\beta \geq \frac{4\sigma^2 + \gamma^2}{2d} \left(1+\frac{4\sigma^2}{\gamma^2}\right)^{\frac{d}{2}} + & \frac{n(d\log(2dn) + \|\theta^*\|^2 - d - d\log(4\sigma^2+\gamma^2))}{2} \\
& \quad \quad \quad \quad \quad \quad \quad \quad - \frac{d^2n}{4} \log\left(1+\frac{4\sigma^2}{\gamma^2}\right) .
\end{align*}
When $\gamma^2$ is of order $d$, this leads to a bound of order (up to a logarithmic factor):
$$
\textcolor{black}{n} \max\big(\|\theta^*\|^2,d\big) \lesssim \beta ,
$$
and we recover the bound that we found for the exact MMD-Bayes.

\section{Projected Stochastic Gradient Algorithm for VI.}\label{apd:algorithm}

In this section, we provide details of a stochastic gradient algorithm (PSGAVI) to compute the Gaussian mean-field approximation, with a necessary projection step if $\mathcal{M}\subsetneq \mathbb{R}^d$ and $\mathcal{S}\subsetneq \mathbb{R}^d_{>0}$. We assume that $\mathcal{M}\subset\mathbb{R}^d$ and $\mathcal{S}\subset\mathbb{R}^d_{>0}$ are closed and convex sets so that the orthogonal projection $\Pi_{\mathcal{M}}$ on $\mathcal{M}$ and $\Pi_{\mathcal{S}}$ on $\mathcal{S}$ are well-defined. We choose a standard Gaussian prior $\pi=\mathcal{N}(0,1)$.

Another important assumption is that the model is generative, i.e that one can easily sample from distributions belonging to the model $\{P_\theta,\theta\in\Theta\}$. The main idea of the algorithm \citep{Roy2015,li15,Briol2019} is then to approximate the gradient of the criterion to minimize $\textnormal{KL}(\mathcal{N}(m,\textnormal{diag}(s^2))\|\pi_n^\beta)$ using an unbiased U-statistic estimate based on random samples from the generative model, and to use a projected stochastic gradient algorithm. We recall that we use the componentwise multiplication.

\subsection*{Criterion to minimize:}

As explained in Appendix \ref{apd:extendedpriormass}, the optimization program is equivalent to minimizing:
$$
\underset{(m,s) \in \mathcal{M}\times\mathcal{S}}{\arg\min} \, \bigg\{  \int \mathbb{D}_k^{\textcolor{black}{2}}(P_{\theta},\hat{P}_n) \mathcal{N}({\rm d}\theta|m,\textnormal{diag}(s^2)) + \frac{1}{2\beta} \sum_{j=1}^d \bigg[ m_j^2 + s_j^2 - \log(s_j^2) - 1 \bigg] \bigg\} .
$$
We know that: $$ 
\mathbb{D}_k^{\textcolor{black}{2}}(P_\theta,\hat{P}_n)
= \mathbb{E}_{X,X'\sim P_\theta} [k(X,X')]
- \frac{2}{n}\sum_{i=1}^n \mathbb{E}_{X\sim P_\theta}[k(X_i,X)]
+  \frac{1}{n^2}\sum_{1\leq i,j\leq n} k(X_i,X_j) .
$$

Hence, the criterion to minimize is:
\begin{align*}
R_n(m,s) & := \int \mathbb{E}_{X,X'\sim P_\theta} [k(X,X')] \mathcal{N}({\rm d}\theta|m,\textnormal{diag}(s^2)) \\
& - \int \frac{2}{n}\sum_{i=1}^n \mathbb{E}_{X\sim P_\theta}[k(X_i,X)] \mathcal{N}({\rm d}\theta|m,\textnormal{diag}(s^2)) + \frac{1}{2\beta} \sum_{j=1}^d \bigg\{ m_j^2 + s_j^2 - \log(s_j^2) - 1 \bigg\} .
\end{align*}

\subsection*{Gradient computation:}

The first-order gradient algorithm PSGAVI requires the computation of the gradient of the criterion $R_n$ with respect to $m$ and $s$. In the following, we will use componentwise operations.

The expression of $R_n$ contains two terms that can be written as $\int f(\theta) \mathcal{N}({\rm d}\theta|m,\textnormal{diag}(s^2))$, and the derivative of this expectation can be hard to evaluate. We use the so-called reparameterization trick which is very popular in the variational inference community and approximate the expectation by a stochastic gradient estimator:
$$
\int \nabla_m f(m+s\theta) \mathcal{N}({\rm d}\theta|0,I_d) \approx \frac{1}{M} \sum_{k=1}^M \nabla_m f(m+s\theta^k)
$$
and
$$
\int \nabla_{s}f(m+s\theta) \mathcal{N}({\rm d}\theta|0,I_d) \approx \frac{1}{M} \sum_{k=1}^M \nabla_{s} f(m+s\theta^k)
$$
where $M$ denotes the number of samples $\theta^k$ drawn from the standard Gaussian.

Hence, the gradients of the criterion are:
\begin{align*}
\nabla_m R_n(m,s) \approx \frac{1}{M} \sum_{k=1}^M \nabla_m \mathbb{E}_{X,X'\sim P_{m+s\theta^k}} &[k(X,X')] \\
& - \frac{1}{M} \frac{2}{n} \sum_{k=1}^M \sum_{i=1}^n \nabla_m \mathbb{E}_{X\sim P_{m+s\theta^k}}[k(X_i,X)] + \frac{1}{\beta} \cdot m ,
\end{align*}
\begin{align*}
\nabla_{s} R_n(m,s) \approx \frac{1}{M} \sum_{k=1}^M \nabla_{s} &\mathbb{E}_{X,X'\sim P_{m+s\theta^k}} [k(X,X')] \\
& - \frac{1}{M} \frac{2}{n} \sum_{k=1}^M \sum_{i=1}^n \nabla_{s} \mathbb{E}_{X\sim P_{m+s\theta^k}}[k(X_i,X)] + \frac{1}{\beta} ( s - s^{-1} ) .
\end{align*}

Moreover, using the log-derivative trick for differentiable log-densities:
$$
\nabla_\theta \mathbb{E}_{X,X'\sim P_{\theta}} [k(X,X')] = 2 \mathbb{E}_{X,X'\sim P_\theta} \bigg[ k(X,X') \nabla_\theta[\log p_\theta(X) ] \bigg] ,
$$
$$
\nabla_\theta \mathbb{E}_{X\sim P_{\theta}}[k(X_i,X)] = \mathbb{E}_{X\sim P_\theta} \bigg[ k(X_i,X) \nabla_\theta[\log p_\theta(X) ] \bigg] .
$$
Hence, we obtain stochastic gradients using i.i.d samples $(Y_1,\dots,Y_M)$ from $P_\theta$:
\begin{align*}
\widehat{\nabla_m R_n}(m,s) = \frac{2}{M^2} \sum_{k=1}^M \sum_{j=1}^M \bigg\{ \frac{1}{M-1} \sum_{\ell \neq j} k(Y_j,Y_\ell) - \frac{1}{n} \sum_{i=1}^nk(X_i,Y_j) \bigg\} \nabla_m&[\log p_{m+s\theta^k}(Y_j) ] \\
& + \frac{1}{\beta} \cdot m 
\end{align*}
and
\begin{align*}
\widehat{\nabla_s R_n}(m,s) = \frac{2}{M^2} \sum_{k=1}^M \sum_{j=1}^M \bigg\{ \frac{1}{M-1} \sum_{\ell \neq j} k(Y_j,Y_\ell) - \frac{1}{n} \sum_{i=1}^nk(X_i,Y_j) \bigg\} \nabla_{s}&[\log p_{m+s\theta^k}(Y_j)] \\
& + \frac{1}{\beta} ( s - s^{-1} ) .
\end{align*}

Note that when the log-density $\log p_\theta(x)$ is not differentiable, it is often possible to compute the stochastic gradients involving $\theta^1,...,\theta^M$ directly, without using the Monte Carlo samples $Y_1,...,Y_M$. For instance, when the model is a uniform distribution $P_\theta=\mathcal{U}([\theta-a,\theta+a])$ and when the kernel can be written as $k(x,y)=K(x-y)$ for some function $K$ (such as Gaussian kernels), we have:
$$
\mathbb{E}_{X,X'\sim P_{\theta}} [k(X,X')] = \int_{\theta-a}^{\theta+a} \int_{\theta-a}^{\theta+a} K(x-x') dx dx' = \int_{-a}^{+a} \int_{-a}^{+a} K(x-x') dx dx' ,
$$
and
$$
\mathbb{E}_{X\sim P_{\theta}}[k(X_i,X)] = \int_{\theta-a}^{\theta+a} K(x-X_i) dx = \int_{\theta-a-X_i}^{\theta+a-X_i} K(x) dx .
$$
Hence,
$$
\nabla_m \mathbb{E}_{X,X'\sim P_{m+s\theta^k}} [k(X,X')] = 0 ,
$$
$$
\nabla_s \mathbb{E}_{X,X'\sim P_{\theta}} [k(X,X')] = 0 ,
$$
and
$$
\nabla_m \mathbb{E}_{X\sim P_{m+s\theta^k}}[k(X_i,X)] = K(m+s\theta^k+a-X_i) - K(m+s\theta^k-a-X_i) ,
$$
$$
\nabla_s \mathbb{E}_{X\sim P_{m+s\theta^k}}[k(X_i,X)] = s K(m+s\theta^k+a-X_i) - s K(m+s\theta^k-a-X_i) .
$$



\vspace{0.2cm}

\subsection*{PSGAVI algorithm:}

The Projected Stochastic Gradient Algorithm for Variational Inference is the following:

\begin{algorithm2e}
\caption{PSGAVI}
\label{alg:psgavi}
\KwIn{A dataset $(X_1,...,X_n)$, a model $\{P_\theta,\theta\in\Theta\subset\mathbb{R}^d\}$, a kernel $k$, a sequence of steps $(\eta_t)_{t\geq 1}$, a batch size $M$, a stopping time $T$, closed and convex sets $\mathcal{M} \subset \mathbb{R}^d$ and $\mathcal{S} \subset \mathbb{R}^d_{>0}$, an initial mean $m^{(0)}\in\mathcal{M}$, an initial covariance matrix $\text{diag}(s^{(T)2})$ where $s^{(0)} \in \mathcal{S}$.}
\KwOut{A variational Gaussian density $\mathcal{N}(\theta|m^{(T)},\text{diag}(s^{(T)2}))$}
\For{$t\leftarrow 1$ \KwTo $T$}{
  draw $(Y_1,\dots,Y_M)$ i.i.d from $P_{m^{(t-1)}}$\;
  $m^{(t)} = \Pi_{\mathcal{M}} \left( m^{(t-1)} - \eta_t \widehat{\nabla_m R_n}(m^{(t-1)},s^{(t-1)}) \right) $\;
  $s^{(t)} = \Pi_{\mathcal{S}} \left( s^{(t-1)} - \eta_t \widehat{\nabla_s R_n}(m^{(t-1)},s^{(t-1)}) \right) $\;
}
\end{algorithm2e}

\textcolor{black}{A theoretical analysis of the algorithm, in the spirit of~\cite{cheriefonline}, goes beyond the scope of this paper and will be the object of future works.}

\section{Numerical simulations.}\label{apd:simulations}

In this section, we provide numerical experiments that support our theoretical results. We studied three different and simple problems: the robust unidimensional Gaussian mean estimation, the robust multidimensional Gaussian mean estimation, and the uniform location parameter estimation.

In each experiment, we compared the mean of the variational approximation of the MMD-Bayes to other estimators: the median estimator and the MLE in the Gaussian mean estimation problem, i.e the componentwise median and the arithmetic mean, and the method of moments and the MLE in the uniform location parameter estimation problem, i.e the arithmetic mean and the average between the largest and the lowest values. We chose a value of $\beta$ of $e^{nd}$, a number of Monte-Carlo samples equal to $n$ and a step-size of $\eta_t=1/\sqrt{t}$. We used the Gaussian kernel $k(x,y)=e^{-\|x-y\|^2/d}$ where $d$ is the dimension and we repeated each experiment $100$ times. 

\vspace{0.2cm}
\textbf{Gaussian mean estimation problem:}
for both the uni- and the multidimensional cases, we randomly sampled $n=200$ i.i.d observations from a Gaussian distribution $\mathcal{N}(\theta,I_d)$ where $I_d$ is the identity matrix of dimension $d$ and $\theta$ is the vector with all components equal to $2$. Some proportion $\epsilon \in [0,0.2]$ of corrupted observations is replaced by independent samples which components are independently sampled from a standard Cauchy distribution $\mathcal{C}(0,1)$. We compared the mean of the variational approximation with the MLE (i.e the arithmetic mean) and the componentwise median using the squared root of the MSE.

\begin{figure}[htbp]
\floatconts
  {fig:nodes}
  {}
  {\includegraphics[width=1.0\linewidth]{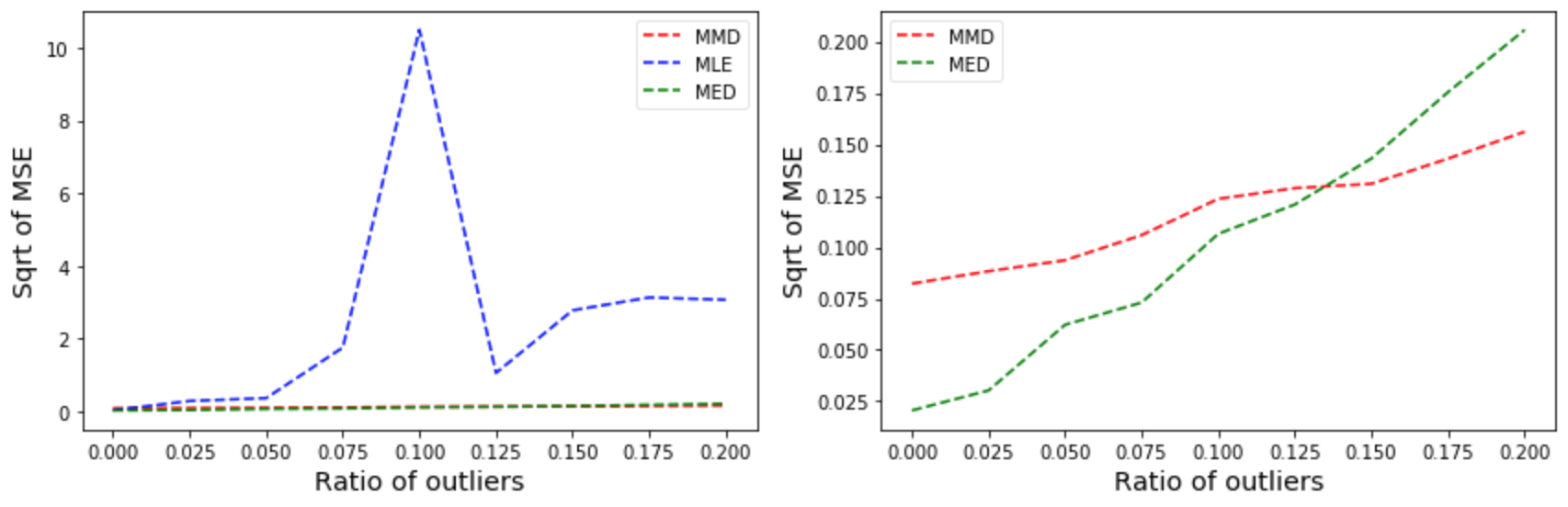}}
  \vspace{-0.5cm}
  {Figure 2 - Comparison of the square root of the MSE for the MMD estimator, the MLE and the median in the robust multidimensional Gaussian mean estimation problem for various values of the proportion of outliers. Here $d=15$.}
\end{figure}

\textbf{Uniform location parameter estimation problem:}
we randomly sampled $n=200$ i.i.d observations from a uniform distribution $\mathcal{U}\big([\theta-\frac{1}{2},\theta+\frac{1}{2}]\big)$ where $\theta=1$. Following the previous set of experiments, the proportion $\epsilon \in [0,0.2]$ of data is replaced by outliers from a Gaussian $\mathcal{N}(20,1)$. We compared the mean of the variational approximation with the MLE (i.e the average between the largest and the lowest values) and the method of moments estimator (i.e the arithmetic mean) using again the square root of the MSE.

\begin{figure}[htbp]
\floatconts
  {fig:nodes}
  {}
  {\includegraphics[width=0.5\linewidth]{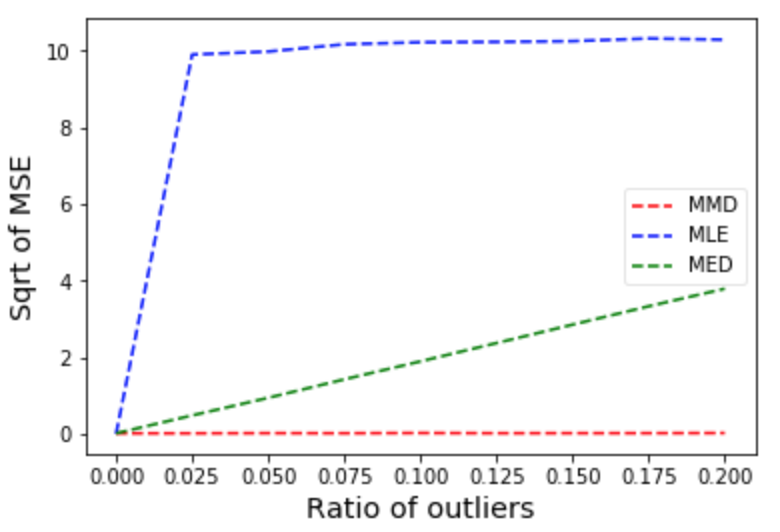}}
  \vspace{-0.5cm}
  {Figure 3 - Comparison of the square root of the MSE for the MMD estimator, the MLE and the method of moments in the robust estimation of the location parameter of a uniform distribution for various values of the proportion of outliers.}
\end{figure}

\textbf{Results:}
The error of our estimators as a function of the contamination ratio $\epsilon$ is plotted in Figures 1, 2 and 3. These plots show that our method is applicable to various problems and leads to a good estimator for all of them. Indeed, the plots in Figures 1 and 2 show that the MSE for the MMD estimator performs as well as the componentwise median and even better when the number of outliers in the dataset increases, much better than the MLE in the robust Gaussian mean estimation problem, and is not affected that much by the presence of outliers in the data. For the uniform location parameter estimation problem addressed in Figure 3, the MMD estimator is clearly the one that performs the best and is not affected by a reasonable proportion of outliers, contrary to the method of moments which square root of MSE is increasing linearly with $\epsilon$ and to the MLE that gives inconsistent estimates as soon as there is an outlier in the data.

\end{document}